\documentclass[11pt, reqno, psamsfonts]{amsart}
\pdfoutput=1
\usepackage{amssymb}
\usepackage{amsthm}
\usepackage{amsmath}
\usepackage{latexsym}
\usepackage[T1]{fontenc}
\usepackage[utf8]{inputenc}
\usepackage[russian, french, english]{babel}
\usepackage{graphicx}
\usepackage[justification=centering, labelfont=bf]{caption}
\usepackage{mathtools}
\usepackage[hidelinks]{hyperref}
\usepackage{enumitem}
\usepackage{mathrsfs}
\usepackage{framed}
\usepackage{tikz}
\usepackage{bm}

\usepackage[shortcuts]{extdash}

\usepackage[foot]{amsaddr}

\usepackage[spacing=true,kerning=true,babel=true,tracking=true]{microtype}

\usetikzlibrary{shapes,snakes}
\usetikzlibrary{arrows.meta}

\usepackage[backend=biber, style=alphabetic, sorting=nyt, maxnames=100]{biblatex}
\usepackage[left=2.3cm,right=2.3cm,top=2.3cm,bottom=2.3cm,bindingoffset=0cm]{geometry}

\title{The Johansson--Molloy Theorem for DP-Coloring}
\date{}
\author{Anton~Bernshteyn}
\address{Department of Mathematics, University of Illinois at Urbana--Champaign, IL, USA}
\email{bernsht2@illinois.edu}
\thanks{This research is partially supported by the Illinois Distinguished Fellowship.}

\newtheorem{theo}{Theorem}[section]

\newtheorem{lemma}[theo]{Lemma}
\newtheorem{corl}[theo]{Corollary}

\newtheorem*{claim*}{Claim}
\newtheorem*{lemma*}{Lemma}

\theoremstyle{definition}
\newtheorem{defn}[theo]{Definition}

\newtheorem*{assum}{Standing assumptions}

\theoremstyle{remark}
\newtheorem*{remk}{Remark}

\newcommand*{\myproofname}{Proof}
\newenvironment{claimproof}[1][\myproofname]{\begin{proof}[#1]}{\end{proof}}

\newcommand{\0}{\varnothing}
\newcommand{\set}[1]{\{#1\}}
\newcommand{\dom}{\mathrm{dom}}

\newcommand{\N}{\mathbb{N}}

\renewcommand{\epsilon}{\varepsilon}
\renewcommand{\phi}{\varphi}
\renewcommand{\theta}{\vartheta}
\renewcommand{\tilde}{\widetilde}
\renewcommand{\leq}{\leqslant}
\renewcommand{\geq}{\geqslant}
\newcommand{\powerset}[1]{\operatorname{Pow}(#1)}

\newcommand{\defeq}{\coloneqq}
\newcommand{\Cov}[1]{\mathscr{#1}}
\newcommand{\event}[1]{\{{#1}\}}
\renewcommand{\mathbf}[1]{{\bm{#1}}}

\numberwithin{equation}{section}

\makeatletter
\newcommand{\neutralize}[1]{\expandafter\let\csname c@#1\endcsname\count@}
\makeatother

\newcommand{\bemph}[1]{{\upshape#1}} % define how emphasised brackets should look
\newcommand{\ep}[1]{\bemph{(}#1\bemph{)}} % parentheses

\renewbibmacro{in:}{}

\renewbibmacro*{volume+number+eid}{%
	\printfield{volume}%
	\setunit*{\addnbspace}% NEW (optional); there's also \addnbthinspace
	\printfield{number}%
	\setunit{\addcomma\space}%
	\printfield{eid}}
\DeclareFieldFormat[article]{volume}{\textbf{#1}\space}
\DeclareFieldFormat[article]{number}{\mkbibparens{#1}}

\DeclareFieldFormat{journaltitle}{#1,}
\DeclareFieldFormat[thesis]{title}{\mkbibemph{#1}\addperiod}
\DeclareFieldFormat[article, unpublished, thesis]{title}{\mkbibemph{#1},}
\DeclareFieldFormat[book]{title}{\mkbibemph{#1}\addperiod}
\DeclareFieldFormat[unpublished]{howpublished}{#1, }

\DeclareFieldFormat{pages}{#1}

\DeclareFieldFormat[article]{series}{Ser.~#1\addcomma}

\begin{document}
	
	\maketitle
	
	\begin{abstract}
		The aim of this note is twofold. On the one hand, we present a streamlined version of Molloy's new proof of the bound $\chi(G) \leq (1+o(1))\Delta(G)/\ln \Delta(G)$ for triangle-free graphs $G$, avoiding the technicalities of the entropy compression method and only using the usual ``lopsided'' Lov\'asz Local Lemma (albeit in a somewhat unusual setting). On the other hand, we extend Molloy's result to DP\=/coloring (also known as correspondence coloring), a generalization of list coloring introduced recently by Dvo\v{r}\'{a}k and Postle.
	\end{abstract}
	
		\subsection*{Basic notation}
		
		All graphs considered here are finite, undirected, and simple. For a graph $G$, its vertex and edge sets are denoted $V(G)$ and $E(G)$ respectively.
		
		For a subset $U \subseteq V(G)$, $\overline{U} \defeq V(G) \setminus U$ is the complement of $U$ and $G[U]$ is the subgraph of $G$ induced by~$U$. Let $G - U \defeq G[\overline{U}]$. For two subsets $U_1$, $U_2 \subseteq V(G)$, $E_G(U_1, U_2)$ denotes the set of all edges of $G$ with one endpoint in $U_1$ and the other one in $U_2$.
		
		For $u \in V(G)$, $N_G(u)$ denotes the set of all neighbors of $u$ in $G$, and $\deg_G(u) \defeq |N_G(u)|$ denotes the degree of $u$. Let $N_G[u] \defeq N_G(u) \cup \set{u}$ be the closed neighborhood of $u$. For $d \in \N$, $N_G^d[u]$ denotes the set of all vertices that are at distance at most $d$ from $u$ (thus, $N_G^0[u] = \set{u}$ and $N_G^1[u] = N_G[u]$). The maximum degree of $G$ is denoted $\Delta(G)$.
		
		Given a subset $U \subseteq V(G)$, let $N_G(U) \defeq \bigcup_{u \in U} N_G(u)$ and $N_G[U] \defeq \bigcup_{u \in U} N_G[u]$. A set $I \subseteq V(G)$ is \emph{independent} if $I \cap N_G(I) = \0$, i.e., if there are no $u$, $v \in I$ with $uv \in E(G)$.
		
		For a set $S$, $\powerset{S}$ denotes the power set of $S$, i.e., the set of all subsets of $S$.
	
	\section{Introduction}
	
	\noindent A \emph{proper coloring} of a graph~$G$ is a function $f \colon V(G) \to C$, where $C$ is a set, whose elements are referred to as \emph{colors}, such that $f(u) \neq f(v)$ for each edge $uv \in E(G)$. The smallest $k \in \N$ such that there exists a proper coloring $f \colon V(G) \to C$ with $|C| = k$ is called the \emph{chromatic number} of $G$ and is denoted~$\chi(G)$. \emph{List coloring} is a generalization of ordinary graph coloring that was introduced independently by Vizing~\cite{Vizing} and Erd\H{o}s, Rubin, and Taylor~\cite{ERT}. A~\emph{list assignment} for $G$ is a function $L \colon V(G) \to \powerset{C}$. %For each $u \in V(G)$, the set $L(u)$ is called the \emph{list} of $u$ and its elements are said to be \emph{available} for~$u$.
	%If $|L(u)| = k$ for all $u \in V(G)$, then~$L$ is called a \emph{$k$-list assignment}.
	A proper coloring $f \colon V(G) \to C$ is called an \emph{$L$-coloring} if $f(u) \in L(u)$ for every $u \in V(G)$. The \emph{list-chromatic number} $\chi_\ell(G)$ of $G$  is the smallest $k \in \N$ such that $G$ admits an $L$\=/coloring for every list assignment $L$ such that $|L(u)| = k$ for all $u \in V(G)$.

	The following is a celebrated result of Johansson~\cite{Joh96}:
	
	\begin{theo}[{Johansson~\cite{Joh96}}]\label{theo:Joh}
		There exists a positive constant $C$ such that for every triangle-free graph~$G$ with maximum degree $\Delta$,
		\[
			\chi_\ell(G) \leq (C+o(1)) \frac{\Delta}{\ln \Delta}.
		\]
	\end{theo}
	
	\begin{remk}
		Throughout, we use $o(1)$ to indicate a function of $\Delta$ that approaches $0$ as $\Delta \to \infty$.
	\end{remk}

	Johansson originally proved Theorem~\ref{theo:Joh} with $C = 9$. Subsequently, Pettie and Su~\cite{PS15} improved the bound to $C = 4$. Very recently, Molloy~\cite{Mol17} reduced the constant to $C = 1$:
	
	\begin{theo}[{Molloy~\cite[Theorem~1]{Mol17}}]\label{theo:Mol}
		For every triangle-free graph~$G$ with maximum degree~$\Delta$,
		\[
		\chi_\ell(G) \leq (1+o(1)) \frac{\Delta}{\ln \Delta}.
		\]
	\end{theo}
	
	The two main new ideas that allowed Molloy to dramatically simplify Johansson's proof and establish Theorem~\ref{theo:Mol} are:
	\begin{itemize}
		\item[--] a new coupon collector\=/type result \cite[Lemma~7]{Mol17} with elements drawn uniformly at random from possibly distinct sets; and
		\item[--] the use of the entropy compression method instead of iterated applications of the Lov\'{a}sz Local Lemma.
	\end{itemize}
	%a new coupon collector\=/type result \cite[Lemma~6]{Mol17} with elements drawn uniformly at random from possibly distinct sets and (\emph{b}) the use of the entropy compression method instead of iterated applications of the Lov\'{a}sz Local Lemma.
	
	The entropy compression method (the name is due to Tao~\cite{Tao}) was developed by Moser and Tardos in order to prove an algorithmic version of the Lov\'asz Local Lemma. Later it was observed (first by Grytczuk, Kozik, and Micek in their study of nonrepetitive sequences~\cite{Grytczuk}) that the  Moser--Tardos algorithmic approach can sometimes lead to improved combinatorial results if applied directly, with no explicit mention of the Lov\'asz Local Lemma. This technique has since found many applications, especially in the study of graph coloring; see, e.g., \cite{Esperet, BCGR, Duj}. The recent results of Molloy are a part of this program.
	
	One may wonder however why the entropy compression method should be significantly superior to the Local Lemma when applied specifically to the problem of coloring triangle-free graphs. Indeed, there is a lot of ``slackness'' in the way the Local Lemma is used in Johansson's proof of Theorem~\ref{theo:Joh}: certain events happen with exponentially small probabilities, even though a polynomial upper bound would have sufficed. In other words, the ``bottleneck'' in the proof is not the Local Lemma \emph{per se}, but rather some expectation/concentration details. Thus, it may appear surprising that using a better alternative to the Local Lemma leads to improvements in this particular case.
	
	In this note we show that the intuition outlined in the previous paragraph is, in fact, accurate: one can replace the entropy compression method in Molloy's proof of Theorem~\ref{theo:Mol} by the usual Local Lemma. This makes the argument particularly short and straightforward, as it removes the need for the technical analysis of a randomized recoloring procedure.
	
	The main novelty in our version of the proof consists in choosing a partial proper coloring $f$ of~$G$ \emph{uniformly at random} (see~Lemma~\ref{lemma:main}). Note that the colors of individual vertices under $f$ are highly dependent, so understanding the behavior of $f$ at first appears rather difficult. That is why one usually assigns colors to the vertices of~$G$ \emph{independently} from each other. But independence comes at a price: It is impossible to ensure that the resulting coloring is proper away from a very small part of the graph. This necessitates an iterative approach, forcing one to repeat the procedure several times until a sufficiently large proportion of the vertices has been colored. Our main observation is that, despite the dependencies, it is still possible to use the Local Lemma to directly analyze a uniformly random partial proper coloring, thus obviating the need for iteration.
	
	Using the Local Lemma instead of the entropy compression is the only significant difference between our argument and the original proof of Theorem~\ref{theo:Mol} due to Molloy. In particular, we need a coupon collector\=/type lemma (Lemma~\ref{lemma:conc}), which is, essentially, a rephrasing of \cite[Lemma~7]{Mol17}. Nevertheless, to make the presentation self-contained, we include all (or most of) the details.
	
	Our second contribution is verifying the conclusion of Theorem~\ref{theo:Mol} in the context of DP-coloring (also known as \emph{correspondence coloring})---a generalization of list coloring introduced recently by Dvo\v{r}\'{a}k and Postle~\cite{DP} (see Section~\ref{sec:DP} for the definitions). A version of Johansson's theorem for DP-coloring was established previously by the author in~\cite[Theorem~1.7]{Ber}, with no attempt to optimize the constant factor.
	
	\begin{theo}\label{theo:main}
		For every triangle-free graph~$G$ with maximum degree~$\Delta$,
		\[
		\chi_{DP}(G) \leq (1+o(1)) \frac{\Delta}{\ln \Delta}.
		\]
	\end{theo}

	Theorem~\ref{theo:main} yields a rather tight bound on the DP-chromatic number of triangle-free regular graphs (the lower bound is given by~\cite[Theorem~1.6]{Ber}; cf.~\cite[Corollary~1.8]{Ber}):
	
	\begin{corl}
		For every $\Delta$-regular triangle-free graph $G$,
		\[
			(1/2-o(1))\frac{\Delta}{\ln \Delta} \leq \chi_{DP}(G) \leq (1+o(1)) \frac{\Delta}{\ln \Delta}.
		\]
	\end{corl}

	Another result of Johansson asserts that $\chi_\ell(G) = O(\Delta(G) \ln \ln \Delta(G) / \ln \Delta(G))$ if $G$ is $K_r$-free for some fixed $r\geq 4$. Molloy~\cite[Theorem~2]{Mol17} also gave a new short proof of this bound with explicit dependence on $r$. In Section~\ref{sec:Kr}, we extend it to DP-coloring:
	
	\begin{theo}\label{theo:Kr}
		There exists a positive constant $C$ such that for any $r \geq 4$ and for every $K_r$-free graph~$G$ with maximum degree $\Delta$,
		\[
			\chi_{DP}(G) \leq C\frac{r\Delta \ln \ln \Delta}{\ln \Delta}.
		\]
	\end{theo}
	
	We make no attempt to optimize the constant factor in Theorem~\ref{theo:Kr}. It is conjectured~\cite[Conjecture~3.1]{AKS} that the correct upper bound for fixed $r$ should be of the order $O(\Delta/\ln \Delta)$.
	
	The main technical step in the proof of Theorem~\ref{theo:Kr} is Lemma~\ref{lemma:Kr_bounds}. It is similar to \cite[Lemma~15]{Mol17}; however, it is necessary to modify the proof of \cite[Lemma~15]{Mol17} somewhat in order to adapt it for the DP-coloring framework, since, in contrast to list coloring, a DP-coloring of a graph cannot be naturally represented as a partition of its vertex set into independent subsets.
	
	\section{DP-Coloring}\label{sec:DP}
	
	\noindent DP-coloring was introduced by Dvo\v{r}\'{a}k and Postle in~\cite{DP} in order to settle a long-standing open question of Borodin regarding list coloring planar graphs with no cycles of certain lengths~\cite[Problem~8.1]{Bor13}. Just like list coloring extends ordinary coloring by allowing the lists of available colors to vary from vertex to vertex, DP-coloring generalizes list coloring by allowing the identifications between the colors in the lists to vary from edge to edge.
	
	\begin{defn}\label{defn:cover}
		Let $G$ be a graph. A \emph{cover} of $G$ is a pair $\Cov{H} = (L, H)$, consisting of a graph $H$ and a function $L \colon V(G) \to \powerset{V(H)}$, satisfying the following requirements:
		\begin{enumerate}[labelindent=\parindent,leftmargin=*,label=(C\arabic*)]
			\item the sets $\set{L(u) \,:\,u \in V(G)}$ form a partition of $V(H)$;
			\item for every $u \in V(G)$, the graph $H[L(u)]$ is complete;
			\item if $E_H(L(u), L(v)) \neq \0$, then either $u = v$ or $uv \in E(G)$;
			\item \label{item:matching} if $uv \in E(G)$, then $E_H(L(u), L(v))$ is a matching.
		\end{enumerate}
		A cover $\Cov{H} = (L, H)$ of $G$ is \emph{$k$-fold} if $|L(u)| = k$ for all $u \in V(G)$.
	\end{defn}
	
	\begin{remk}
		The matching $E_H(L(u), L(v))$ in Definition~\ref{defn:cover}\ref{item:matching} does not have to be perfect and, in particular, is allowed to be empty.
	\end{remk}
	
	%The second cover of $C_4$ shown in Fig.~\ref{fig:cycle} proves that $\chi_{DP}(C_4) \geq 3$.
	
	%Fig.~\ref{fig:cycle} shows an example of two distinct $2$-fold covers of $G \cong C_4$.
	
	\begin{defn}
		Let $\Cov{H} = (L, H)$ be a cover of a graph $G$. An \emph{$\Cov{H}$-coloring} of $G$ is an independent set in $H$ of size $|V(G)|$.
	\end{defn}
	
	\begin{remk}\label{remk:single}
		Equivalently, an independent set $I$ in $H$ is an $\Cov{H}$-coloring if $|I \cap L(u)| = 1$ for all $u \in V(G)$.
	\end{remk}
	
	\begin{defn}
		The \emph{DP-chromatic number} $\chi_{DP}(G)$ of a graph $G$ is the smallest $k \in \N$ such that~$G$ admits an $\Cov{H}$-coloring for every $k$-fold cover $\Cov{H}$ of $G$.
	\end{defn}
	
	%\begin{exmp}
	Fig.~\ref{fig:cycle} shows a pair of distinct $2$-fold covers of the $4$-cycle $C_4$. Note that $C_4$ admits an $\Cov{H}_1$\=/coloring but not an $\Cov{H}_2$\=/coloring. In particular, $\chi_{DP}(C_4) \geq 3$. On the other hand, it can be easily seen that $\chi_{DP}(G) \leq \Delta(G) + 1$ for any graph $G$, and so $\chi_{DP}(C_4) = 3$. A similar argument demonstrates that $\chi_{DP}(C_n) = 3$ for any cycle $C_n$ of length $n \geq 3$.
	%\end{exmp}
	
	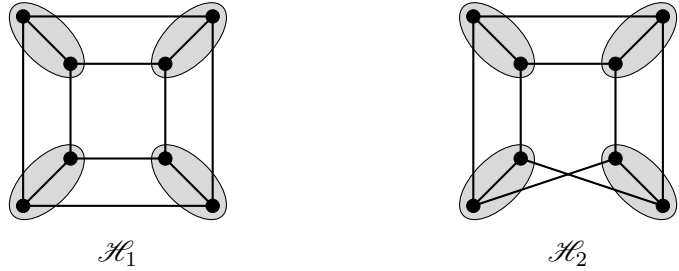
\begin{figure}[h]
		\centering	
		\begin{tikzpicture}[scale=0.63]
		\definecolor{light-gray}{gray}{0.85}
		
		\filldraw[fill=light-gray]
		(6.5,0) circle [x radius=1cm, y radius=5mm, rotate=45]
		(6.5,3) circle [x radius=1cm, y radius=5mm, rotate=-45]
		(9.5,0) circle [x radius=1cm, y radius=5mm, rotate=-45]
		(9.5,3) circle [x radius=1cm, y radius=5mm, rotate=45];
		
		\foreach \x in {(6, -0.5), (6, 3.5), (7, 0.5), (7, 2.5), (9, 0.5), (9, 2.5), (10, -0.5), (10, 3.5)}
		\filldraw \x circle (4pt);
		
		\draw[thick] (6, -0.5) -- (6, 3.5) -- (10, 3.5) -- (10, -0.5) -- cycle;
		
		\draw[thick] (7, 0.5) -- (7, 2.5) -- (9, 2.5) -- (9, 0.5) -- cycle;
		
		\draw[thick] (6, -0.5) -- (7, 0.5) (6, 3.5) -- (7, 2.5) (9, 2.5) -- (10, 3.5) (10, -0.5) -- (9, 0.5);
		
		\node at (8, -1.5) {$\Cov{H}_1$};
		
		\filldraw[fill=light-gray]
		(13+3,0) circle [x radius=1cm, y radius=5mm, rotate=45]
		(13+3,3) circle [x radius=1cm, y radius=5mm, rotate=-45]
		(16+3,0) circle [x radius=1cm, y radius=5mm, rotate=-45]
		(16+3,3) circle [x radius=1cm, y radius=5mm, rotate=45];
		
		\foreach \x in {(12.5+3, -0.5), (12.5+3, 3.5), (13.5+3, 0.5), (13.5+3, 2.5), (15.5+3, 0.5), (15.5+3, 2.5), (16.5+3, -0.5), (16.5+3, 3.5)}
		\filldraw \x circle (4pt);
		
		\draw[thick] (12.5+3, -0.5) -- (12.5+3, 3.5) -- (16.5+3, 3.5) -- (16.5+3, -0.5) -- (13.5+3, 0.5) -- (13.5+3, 2.5) -- (15.5+3, 2.5) -- (15.5+3, 0.5) -- cycle;
		
		\draw[thick] (12.5+3, -0.5) -- (13.5+3, 0.5) (12.5+3, 3.5) -- (13.5+3, 2.5) (15.5+3, 0.5) -- (16.5+3, -0.5) (15.5+3, 2.5) -- (16.5+3, 3.5);
		
		\node at (14.5+3, -1.5) {$\Cov{H}_2$};
		\end{tikzpicture}
		\caption{Two distinct $2$-fold covers of a $4$-cycle.%Note that $C_4$ is $(L, H_1)$-colorable but not $(L, H_2)$-colorable. In particular, $\chi_{DP}(C_4) \geq 3$ (in fact, $\chi_{DP}(C_4) = 3=\Delta(C_4) +1$).
		}\label{fig:cycle}
	\end{figure}
	
	One can obtain a cover of a graph $G$ from a list assignment for~$G$, thus showing that list coloring is a special case of DP-coloring and, in particular, $\chi_{DP}(G) \geq \chi_\ell(G)$ for all graphs $G$. Specifically, let~$G$ be a graph and let $L \colon V(G) \to \powerset{C}$ be a list assignment for~$G$. Let $H$ denote the graph with vertex set
	$
	V(H) \defeq \set{(u, c)\,:\, u \in V(G) \text{ and } c \in L(u)}
	$,
	in which two distinct vertices $(u, c)$ and $(v, d)$ are adjacent if and only if either $u = v$, or else, $uv \in E(G)$ and $c=d$.
	%$
	%	\text{either } u = v, \qquad \text{or else,}\qquad uv \in E(G) \text{ and } c = d.
	%$
	For each $u \in V(G)$, set
	$
	L'(u) \defeq \set{(u, c) \,:\, c \in L(u)}
	$.
	Then $\Cov{H} \defeq (L', H)$ is a cover of $G$, and there is a natural bijective correspondence between the $L$-colorings and the $\Cov{H}$\=/colorings of $G$. More precisely, if $f$ is an $L$\=/coloring of $G$, then the set
	$
	I_f \defeq \set{(u, f(u)) \,:\, u \in V(G)}
	$
	is an $\Cov{H}$-coloring of $G$; and conversely, given an $\Cov{H}$\=/coloring $I$ of $G$, we have $|I \cap L'(u)| = 1$ for all $u \in V(G)$, so an $L$-coloring $f_I$ can be defined by the property
	$
		(u, f_I(u)) \in I \cap L'(u)
	$.
	
	\section{Proof of Theorem~\ref{theo:main}}\label{sec:proof}
	
	\subsection{Probabilistic tools}
	
	We use the following ``lopsided'' version of the Symmetric Lov\'{a}sz Local Lemma (the~LLL for short); see, e.g., \cite[p.~65]{AS00}:
	
	\begin{lemma}[{\textbf{Lov\'asz Local Lemma}}]\label{lemma:LLL}
		Let $I$ be a finite set. For each $i \in I$, let $B_i$ be a random event. Suppose that for every $i \in I$, there is a set $\Gamma(i) \subseteq I$ such that $|\Gamma(i)| \leq d$ and for all $Z \subseteq I \setminus \Gamma(i)$,
		\[
			\Pr\left[B_i \middle\vert \bigcap_{j \in Z} \overline{B_j}\right] \leq p.
		\] 
		\ep{A horizontal line over an event denotes its negation.} If $4pd \leq 1$, then $\Pr\left[\bigcap_{i \in I} \overline{B_i}\right] > 0$.
	\end{lemma}
	
	\begin{remk}
		In a more commonly used version of the~LLL, each event $B_i$ is mutually independent from the events $B_j$ with $j \not\in \Gamma(i) \cup \set{i}$. Lemma~\ref{lemma:LLL} has exactly the same proof, but it makes no independence requirements (which will be important for its application in the proof of Lemma~\ref{lemma:main}).
	\end{remk}
	
	We will need a version of Chernoff bounds for negatively correlated random variables, introduced by Panconesi and Srinivasan~\cite{PS}. We say that $\set{0,1}$-valued random variables $X_1$, \ldots, $X_n$ are \emph{negatively correlated} if for all $S \subseteq \set{1, \ldots, n}$,
	\[
	\Pr\left[X_i = 1 \text{ for all } i \in S \right] \leq \prod_{i \in S} \Pr\left[X_i = 1\right].
	\]
	
	\begin{lemma}[{\textbf{Chernoff bounds}; see \cite{PS} and \cite[Lemma~3]{Mol17}}]\label{lemma:Chernoff}
		Let $X_1$, \ldots, $X_n$ be $\set{0,1}$-valued random variables and let $Y_i \defeq 1 - X_i$. Set $X \defeq \sum_{i = 1}^n X_i$. If the variables $Y_1$, \ldots, $Y_n$ are negatively correlated, then
		\[
			\Pr\left[X \leq (1 - \delta) \mathbb{E}[X]\right] \leq \exp\left(-\delta^2 \mathbb{E}[X]/2\right) \text{ for any } 0 < \delta < 1.
		\]
		If the variables $X_1$, \ldots, $X_n$ are negatively correlated, then
		\[
			\Pr\left[X \geq (1 + \delta) \mathbb{E}[X] \right] \leq \exp\left(-\delta\mathbb{E}[X]/3\right) \text{ for any } \delta > 1.
		\]
	\end{lemma}
	
	\subsection{Additional notation}\label{subsec:not}
	
	Let $G$ be a graph and let $\Cov{H} = (L,H)$ be a cover of~$G$. For $U \subseteq V(G)$, let $L(U) \defeq \bigcup_{u \in U} L(u)$. Define $H^\ast$ to be the spanning subgraph of $H$ such that an edge $xy \in E(H)$ belongs to $E(H^\ast)$ if and only if $x$ and $y$ are in different parts of the partition $\set{L(u)\,:\, u \in V(G)}$. For clarity, and to emphasize the dependence on $L$, we write $\deg^\ast_{\Cov{H}}(x)$ instead of $\deg_{H^\ast}(x)$. The \emph{domain} of an independent set $I$ in $H$ is $\dom(I) \defeq \set{u \in V(G)\,:\, I\cap L(u) \neq \0}$.  Let $G_I \defeq G - \dom(I)$ and let $\Cov{H}_I = (L_I, H_I)$ denote the cover of $G_I$ defined by
	\[
	H_I \defeq H - N_H[I] \qquad\text{and}\qquad L_I(u) \defeq L(u) \setminus N_H(I) \text{ for all } u \in V(G_I).
	\]
	By definition, if $I'$ is an $\Cov{H}_I$-coloring of $G_I$, then $I \cup I'$ is an $\Cov{H}$-coloring of $G$.
	
	We use bold type to emphasize that a certain value (such as a set $\mathbf{I}$) is a random variable.
	
	\subsection{The proof}
	
	We will reduce Theorem~\ref{theo:main} to the following result from \cite{Hax01},  which is a special case of a more general theorem of Aharoni and Haxell \ep{unpublished}: %, see~\cite{Hax01}:
	
	\begin{theo}[{Aharoni--Haxell;~\cite[Theorem 2]{Hax01}}]\label{theo:Hax}
		Let $\ell \geq 1$ and let $F$ be a graph with $\Delta(F) \leq \ell/2$. Suppose that $V(F) = V_1 \cup \ldots \cup V_n$ is a partition of the vertex set of $F$ such that $|V_i| \geq \ell$ for all $1 \leq i \leq n$. Then $F$ has an independent set $I$ with $|I \cap V_i|=1$ for all $1 \leq i \leq n$. 
	\end{theo}
	
	We will use Theorem~\ref{theo:Hax} in the form of the following corollary:
	
	\begin{lemma}\label{lemma:small_deg}
		Let $\Cov{H} = (L, H)$ be a cover of a graph $G$. If there is a positive integer~$\ell$ such that $|L(u)| \geq \ell$ for all $u \in V(G)$ and $\deg_{\Cov{H}}^\ast(x) \leq \ell/2$ for all $x \in V(H)$, then $G$ is $\Cov{H}$\=/colorable. 
	\end{lemma}
	
	To derive Lemma~\ref{lemma:small_deg}, one simply has to apply Theorem~\ref{theo:Hax} with $H^\ast$ in place of $F$ and the partition $\set{L(u) \,:\, u \in V(G)}$ in place of $\set{V_1, \ldots, V_n}$. For completeness, we give a short self-contained proof of Lemma~\ref{lemma:small_deg} under the stronger assumption that $\deg_{\Cov{H}}^\ast(x) \leq \ell/8$ for all $x \in V(H)$ in the appendix (this weaker version of Lemma~\ref{lemma:small_deg} is also sufficient for our purposes; see~\cite[Theorem~2]{ReedList} and \cite[Lemma~5]{Mol17} for two of its incarnations in the list coloring setting).
	
	%\subsection{The main lemmas}
	
	\begin{assum}
		For the rest of the proof, fix $0<\epsilon<1$, a triangle-free graph $G$ of sufficiently large maximum degree $\Delta$, and a $k$-fold cover $\Cov{H} = (L, H)$ of $G$ with $k = (1+\epsilon) \Delta/\ln\Delta$. Set $\ell \defeq \Delta^{\epsilon/2}$.
	\end{assum}
	
	In view of Lemma~\ref{lemma:small_deg}, it suffices to establish the following:
	
	\begin{lemma}\label{lemma:main}
		\begin{samepage}The graph $H$ contains an independent set $I$ such that:
		\begin{enumerate}[label=\normalfont{(\roman*)}]
			\item $|L_I(u)| \geq \ell$ for all $u \in V(G_I)$; and
			\item $\deg_{\Cov{H}_I}^\ast(x) \leq \ell/2$ for all $x \in V(H_I)$.
		\end{enumerate}
		\end{samepage}
	\end{lemma}
	
	To prove Lemma~\ref{lemma:main}, we need a variant of \cite[Lemma~7]{Mol17}:
	
	\begin{lemma}\label{lemma:conc}
		Fix a vertex $u \in V(G)$ and an independent set $J\subseteq L(\overline{N_G[u]})$. Let $\mathbf{I}'$ be a uniformly random independent subset of $L_J(N_G(u))$ and let $\mathbf{I} \defeq J \cup \mathbf{I}'$. Then:
		\begin{enumerate}[label=\normalfont{(\emph{\alph*})}]
			\item\label{item:a} $\Pr\left[|L_\mathbf{I}(u)| < \ell\right] \leq \Delta^{-3}/8$; and
			\item\label{item:b} $\Pr\left[\text{there is } x \in L_\mathbf{I}(u) \text{ with } \deg^\ast_{\Cov{H}_{\mathbf{I}}}(x) > \ell/2\right] \leq \Delta^{-3}/8$.
		\end{enumerate}
	\end{lemma}
	
	The proof of Lemma~\ref{lemma:conc} is virtually identical to that of \cite[Lemma~7]{Mol17}, so we first show how to derive Lemma~\ref{lemma:main} from Lemma~\ref{lemma:conc} (this is the new ingredient in our version of Molloy's argument).
	
	\begin{proof}[Proof of Lemma~\ref{lemma:main} \ep{assuming Lemma~\ref{lemma:conc}}]
		Choose an independent set $\mathbf{I}$ in $H$ uniformly at random. (Since the domain of $\mathbf{I}$ may be a proper subset of $V(G)$, in the context of list coloring this is equivalent to choosing a uniformly random partial proper coloring.) The following immediate observation plays a key role in the proof:
		\[\label{eq:star}
			\parbox{0.85\textwidth}{
				\emph{Fix $U \subseteq V(G)$ and an independent set $J\subseteq L(\overline{U})$. Then the random variable~$\mathbf{I} \cap L(U)$, conditioned on the event $\event{\mathbf{I} \cap L(\overline{U}) = J}$, is uniformly distributed over the independent subsets of $L_J(U)$.}
			}\tag{$\#$}
		\]
		For each $u \in V(G)$, let $B_u$ denote the event
		\[
			B_{u} \defeq \event{\text{$u \not \in \dom(\mathbf{I})$ and either $|L_I(u)| < \ell$ or there is $x \in L_\mathbf{I}(u)$ with $\deg_{\Cov{H}_\mathbf{I}}^\ast(x) > \ell/2$}}.
		\]
		Clearly, if none of the events $B_u$ happen, then $\mathbf{I}$ satisfies the conclusion of Lemma~\ref{lemma:main}.
		
		For each $u \in V(G)$, set $\Gamma(u) \defeq N_G^3[u]$. Since $|\Gamma(u)| \leq \Delta^3$, to apply the~LLL, it remains to verify that for all~$Z \subseteq \overline{\Gamma(u)}$,
		\[
		\Pr\left[B_{u}\middle\vert \bigcap_{v \in Z} \overline{B_{v}}\right] \leq \Delta^{-3}/4.
		\]
		By definition, the outcome of any $B_{v}$ is determined by the set $\mathbf{I} \cap L(N_G^2[v])$. If $v \not \in \Gamma(u)$, then the distance between $u$ and $v$ is at least $4$, so $N_G^2[v] \subset \overline{N_G(u)}$. Therefore, the set $\mathbf{I}\cap L(\overline{N_G(u)})$ determines the outcome of every~$B_v$ with $v \not \in \Gamma(u)$. Hence, it suffices to prove that
		\[
		\Pr\left[B_{u}\middle\vert \mathbf{I} \cap L(\overline{N_G(u)}) = J\right] \leq \Delta^{-3}/4  \quad \text{for all independent } J \subseteq L(\overline{N_G(u)}).
		\]
		To that end, fix a vertex $u \in V(G)$ and an independent set $J\subseteq L(\overline{N_G(u)})$. We may assume that $u \not \in \dom(J)$, i.e., $J \subseteq L(\overline{N_G[u]})$ (otherwise the event $B_u$ is incompatible with $\event{\mathbf{I}\cap L(\overline{N_G(u)}) = J}$). Let $\mathbf{I}' \defeq \mathbf{I}\cap L(N_G(u))$. By~\eqref{eq:star}, the variable $\mathbf{I}'$, under the condition $\event{\mathbf{I}\cap L(\overline{N_G(u)}) = J}$, is uniformly distributed over the independent subsets of $L_J(N_G(u))$. Thus, we are in the situation described by Lemma~\ref{lemma:conc}, and so
		\[
			\Pr\left[B_{u}\middle\vert \mathbf{I} \cap L(\overline{N_G(u)}) = J\right] \leq \Delta^{-3}/8 + \Delta^{-3}/8 = \Delta^{-3}/4,
		\]
		as desired.
	\end{proof}
	
	\begin{proof}[Proof of Lemma~\ref{lemma:conc}]
	
	Define
	\[
	p_0 \defeq \Pr\left[|L_{\mathbf{I}}(u)| < \ell\right]\;\;\;\;\; \text{ and }\;\;\;\;\;
	p_1 \defeq \Pr\left[\text{there is } x \in L_\mathbf{I}(u) \text{ with } \deg^\ast_{H_{\mathbf{I}}}(x) > \ell/2\right].
	\]
	Let $\bullet$ be a special symbol distinct from all the elements of $V(H)$. Since $G$ is triangle-free, the set $N_G(u)$ is independent, so $E_H(L(v), L(w)) = \0$ for any two distinct $v$, $w \in N_G(u)$. Hence, $\mathbf{I}'$ can be constructed via the following procedure:
	\begin{samepage}
	\begin{leftbar}
		\noindent For each $v \in N_G(u)$, uniformly at random select an element $x_v$ from $L_J(v) \cup \set{\bullet}$.
		\begin{itemize}[label=--]
			\item If $x_v = \bullet$, then leave $\mathbf{I}'\cap L(v)$ empty;
			\item otherwise, set $\mathbf{I}'\cap L(v) \defeq \set{x_v}$.
		\end{itemize}
	\end{leftbar}
	\end{samepage}
	
	For $x \in L(u)$, let $\tilde{N}(x)$ denote the set of all vertices $v \in N_G(u)$ such that $N_H(x) \cap L_J(v) \neq \0$. Using this notation, we obtain
	\[
	\Pr\left[x \in L_\mathbf{I}(u)\right] = \Pr\left[\mathbf{I}' \cap N_H(x) = \0\right] = \prod_{v \in \tilde{N}(x)}  \left(1 - \frac{1}{ |L_J(v)| + 1}\right).
	\]
	Since $|L_J(v)| \geq 1$ for all $v \in \tilde{N}(x)$ and $\exp(-1/\alpha) \leq 1 - 1/(\alpha+1) \leq \exp(-1/(\alpha+1))$ for all $\alpha > 0$,
	\begin{equation}\label{eq:bounds_on_p}
	\exp\left(-\sum_{v \in \tilde{N}(x)}\frac{1}{|L_J(v)|}\right)\leq \Pr\left[x \in L_\mathbf{I}(u)\right] \leq \exp\left(-\sum_{v \in \tilde{N}(x)}\frac{1}{|L_J(v)|+1}\right).
	\end{equation}
	Now we can conclude
	\[
	\mathbb{E}\left[|L_\mathbf{I}(u)|\right] = \sum_{x \in L(u)} \Pr\left[x \in L_\mathbf{I}(u)\right] \geq \sum_{x \in L(u)} \exp\left(-\sum_{v \in \tilde{N}(x)}\frac{1}{|L_J(v)|}\right).
	\]
	Notice that
	\[
	\sum_{x \in L(u)} \sum_{v \in \tilde{N}(x)}\frac{1}{|L_J(v)|} \leq \sum_{\substack{v \in N_G(u)\,:\\ L_J(v) \neq \0}} \, \sum_{y \in L_J(v)} \frac{1}{|L_J(v)|} \leq \deg_G(u) \leq \Delta,
	\]
	so, by the convexity of the exponential function,
	\[
		\sum_{x \in L(u)} \exp\left(-\sum_{v \in \tilde{N}(x)}\frac{1}{|L_J(v)|}\right) \geq k \exp\left(-\frac{\Delta}{k}\right) = \frac{(1+\epsilon)\Delta}{\ln\Delta} \cdot \Delta^{-1/(1+\epsilon)} \geq 2\ell,
	\]
	provided $\Delta$ is large enough. Putting everything together, we obtain $\mathbb{E}\left[|L_\mathbf{I}(u)|\right] \geq 2\ell$. To prove that the random variable $|L_\mathbf{I}(u)|$ is highly concentrated around its expectation, we need the following claim:
	
	\begin{claim*}
		The indicator random variables of the events $\event{x \not\in L_\mathbf{I}(u)}$ for $x \in L(u)$ are negatively correlated.
	\end{claim*}
	\begin{claimproof}[Proof of Claim]
		It is enough to argue that for all $x \in L(u)$ and $Y \subseteq L(u) \setminus \set{x}$, we have
		\begin{equation}\label{eq:cond}
			\Pr\left[x \not\in L_\mathbf{I}(u) \, \middle\vert\, Y \cap L_\mathbf{I}(u) = \0\right] \leq \Pr\left[x \not\in L_\mathbf{I}(u)\right].
		\end{equation}
		To that end, let $x \in L(u)$ and $Y \subseteq L(u) \setminus \set{x}$. Inequality \eqref{eq:cond} is equivalent to
		\[
			\Pr\left[Y \cap L_\mathbf{I}(u) = \0 \, \middle\vert\, x \in L_\mathbf{I}(u) \right] \geq \Pr\left[Y \cap L_\mathbf{I}(u) = \0\right].
		\]
		This can in turn be rewritten as
		\[
			\Pr\left[\mathbf{I}' \cap N_{H^\ast}(y) \neq \0 \text{ for all } y \in Y \, \middle\vert\, \mathbf{I}' \cap N_{H^\ast}(x) = \0 \right] \geq \Pr\left[\mathbf{I}' \cap N_{H^\ast}(y) \neq \0 \text{ for all } y \in Y\right],
		\]
		which holds since the sets $N_{H^\ast}(x)$ and $N_{H^\ast}(Y)$ are disjoint.
	\end{claimproof}
	
	\noindent Now we can apply Lemma~\ref{lemma:Chernoff} to conclude that
	\[
		p_0  \leq \Pr\left[|L_\mathbf{I}(u)| < \frac{1}{2}\mathbb{E}\left[|L_\mathbf{I}(u)|\right]\right] \leq \exp\left(-\frac{1}{8}\mathbb{E}\left[|L_\mathbf{I}(u)|\right]\right) \leq \exp\left(-\ell/4\right)  < \Delta^{-3}/8,
	\]
	for large enough $\Delta$. This proves~\ref{item:a}.
	
	To prove~\ref{item:b}, we will show that for all $x \in L(u)$,
	\[
		p_x \defeq \Pr\left[x \in L_{\mathbf{I}}(u) \text{ and } \deg^\ast_{\Cov{H}_{\mathbf{I}}}(x) > \ell/2\right] \leq \Delta^{-4}.
	\]
	This is enough, as $p_1 \leq \sum_{x \in L(u)} p_x$ and $|L(u)| = k < \Delta/8$ for large enough $\Delta$. Let $x \in L(u)$. The second inequality in~\eqref{eq:bounds_on_p} implies
	\[
		p_x \leq \Pr\left[x \in L_{\mathbf{I}}(u)\right] \leq \exp\left(-\sum_{v \in \tilde{N}(x)}\frac{1}{|L_J(v)|+1}\right),
	\]
	so we may assume
	\[
		\exp\left(-\sum_{v \in \tilde{N}(x)}\frac{1}{|L_J(v)|+1}\right) \geq \Delta^{-4}, \qquad\text{i.e.,}\qquad \sum_{v \in \tilde{N}(x)}\frac{1}{|L_J(v)|+1} \leq 4\ln \Delta.
	\]
	Then
	\[
		\mathbb{E}\left[\deg^\ast_{\Cov{H}_{\mathbf{I}}}(x)\right] = \sum_{v \in \tilde{N}(x)} \Pr\left[v \not\in \dom(\mathbf{I}')\right] = \sum_{v \in \tilde{N}(x)} \frac{1}{|L_J(v)|+1} \leq 4\ln\Delta \leq \ell/4,
	\]
	for large enough $\Delta$. The events $\event{v \not\in \dom(\mathbf{I}')}$ for $v \in \tilde{N}(x)$ are mutually independent, so the Chernoff bound for independent $\set{0,1}$-valued random variables yields
	\[
		p_x \leq \Pr\left[\deg^\ast_{\Cov{H}_{\mathbf{I}}}(x) > \ell/2\right] \leq \Pr\left[\deg^\ast_{\Cov{H}_{\mathbf{I}}}(x) > \mathbb{E}\left[\deg^\ast_{\Cov{H}_{\mathbf{I}}}(x)\right] + \ell/4\right] \leq \exp\left(-\ell/12\right) \leq \Delta^{-4},
	\]
	for large enough $\Delta$, as desired.
	\end{proof}

	\section{Proof of Theorem~\ref{theo:Kr}}\label{sec:Kr}

	\noindent The general scheme of the argument is similar to that of the proof of Theorem~\ref{theo:main}.
	
	\begin{assum}
		Fix an integer $r \geq 4$, a $K_r$-free graph $G$ of large maximum degree $\Delta$, and a $k$-fold cover $\Cov{H} = (L, H)$ of $G$ with $k \geq 200r\Delta\log_2\log_2\Delta/\log_2\Delta$. Set $\ell \defeq \Delta^{9/10}$.
	\end{assum}
	
	The role of Lemma~\ref{lemma:main} is played by the following statement:
	 
	\begin{lemma}\label{lemma:Kr_main}
		The graph $H$ contains an independent set $I$ such that
			\[|L_I(u)| \geq \ell \text{ for all } u \in V(G_I) \qquad \text{and} \qquad \Delta(G_I) < \ell.
			\]
	\end{lemma}
	
	Note that Lemma~\ref{lemma:Kr_main} readily implies Theorem~\ref{theo:Kr}, since the DP-chromatic number of a graph is always at most one plus its maximum degree. Lemma~\ref{lemma:Kr_main} in turn follows from an analog of~\cite[Lemma~15]{Mol17} for DP-coloring:
	
	\begin{lemma}\label{lemma:Kr_bounds}
		Fix a vertex $u \in V(G)$ and an independent set $J\subseteq L(\overline{N_G[u]})$. Let $\mathbf{I}'$ be a uniformly random independent subset of $L_J(N_G(u))$ and let $\mathbf{I} \defeq J \cup \mathbf{I}'$. Then:
		\begin{enumerate}[label=\normalfont{(\emph{\alph*})}]
			\item\label{item:aKr} $\Pr\left[|L_\mathbf{I}(u)| < \ell\right] \leq \Delta^{-3}/8$; and
			\item\label{item:bKr} $\Pr\left[\deg_{G_\mathbf{I}}(u) \geq \ell \text{ and } |L_\mathbf{I}(v)| \geq \ell \text{ for all } v \in N_{G_\mathbf{I}}(u) \right] \leq \Delta^{-3}/8$.
		\end{enumerate}
	\end{lemma}
	
	%Lemma~\ref{lemma:Kr_main} is derived from Lemma~\ref{lemma:Kr_bounds} in the same way as Lemma~\ref{lemma:main} is deduced from Lemma~\ref{lemma:conc}, so we do not spell the proof out here.
	The derivation of Lemma~\ref{lemma:Kr_main} from Lemma~\ref{lemma:Kr_bounds} is almost verbatim identical to that of Lemma~\ref{lemma:main} from Lemma~\ref{lemma:conc}, and we do not spell it out here.
	To prove Lemma~\ref{lemma:Kr_bounds}, we need a variant of a result due to Shearer~\cite{She} that was established by Molloy~\cite[Lemma~14]{Mol17}. For a graph $F$, let $\operatorname{ind}(F)$ denote the number of independent sets in $F$ and let $\overline{\alpha}(F)$ denote the \emph{median} size of an independent set in $F$, i.e., the supremum of all $\alpha \geq 0$ such that $F$ contains at least $\operatorname{ind}(F)/2$ independent sets of size at least $\alpha$. For $\lambda > 0$, define
	\[
		f(\lambda) \defeq \frac{\log_2 \lambda}{2r \log_2\log_2 \lambda}.
	\]
	
	\begin{lemma}[{\cite[Lemma~14]{Mol17}}]\label{lemma:She}
		If $F$ is a nonempty $K_r$-free graph, then $\overline{\alpha}(F) \geq f(\operatorname{ind}(F))$.
	\end{lemma}
	
	\begin{proof}[Proof of Lemma~\ref{lemma:Kr_bounds}]
		We start with the proof of~\ref{item:aKr}. For each $x \in L(u)$, the \emph{layer} of $x$ is the set
		\[
		\Lambda(x) \defeq L_J(N_G(u)) \cap N_H(x).
		\]
		Equivalently, using the notation introduced in \S\ref{subsec:not}, we can write $\Lambda(x) \defeq N_{H^\ast_J}(x)$; i.e., $\Lambda(x)$ contains all the neighbors of $x$ in $H^\ast$ that are not adjacent to any element of $J$.
		Note that the layer of $x$ intersects each list $L(v)$ for $v \in N_G(u)$ in at most one element. Also, for distinct $x$, $y \in L(u)$, the layers $\Lambda(x)$ and $\Lambda(y)$ are disjoint. However, in contrast to the situation in list coloring, $H$ may contain edges between $\Lambda(x)$ and $\Lambda(y)$ that are not covered by the cliques $H[L(v)]$.
		
		Let $x \in L(u)$. For an independent set $Q \subseteq L_J(N_G(u)) \setminus \Lambda(x)$, let $F(x, Q)$ denote the subgraph of $H$ induced by the vertices in $\Lambda(x)$ with no neighbors in $Q$. The following observation is similar to~\eqref{eq:star} from the proof of Lemma~\ref{lemma:main}:
		\[\label{eq:flat}
		\parbox{0.85\textwidth}{
			\emph{Fix $x \in L(u)$ and an independent set $Q \subseteq L_J(N_G(u)) \setminus \Lambda(x)$. Then the random variable $\mathbf{I}' \cap \Lambda(x)$, conditioned on the event $\event{\mathbf{I}' \setminus \Lambda(x) = Q}$, is uniformly distributed over the independent sets in $F(x, Q)$.}
		}\tag{$\flat$}
		\]
		From \eqref{eq:flat}, it follows that $\mathbf{I}'$ can be constructed via the following randomized procedure. Let $x_1$, \ldots, $x_k$ be an arbitrary ordering of the set $L(u)$. For each $1 \leq i \leq k$, let $\Lambda_i \defeq \Lambda(x_i)$.
		
		\begin{leftbar}
			\noindent Let $\mathbf{I}_0$ be a uniformly random independent subset of $L_J(N_G(u))$. Set $\mathbf{s}_0 \defeq 0$ and $\mathbf{t}_0 \defeq 0$.
			
			\noindent Repeat the next steps for each $1 \leq i \leq k$:
			
			\begin{itemize}[label=--]
				\item Let $\mathbf{F}_i \defeq F(x_i, \mathbf{I}_{i-1} \setminus \Lambda_i)$. Define $\mathbf{s}_i$ and $\mathbf{t}_i$ as follows:
				\[
					\begin{array}{c c c l c l c}
					\text{if} & \operatorname{ind}(\mathbf{F}_i) > \Delta^{1/20}, & \text{then} & \mathbf{s}_i \defeq \mathbf{s}_{i-1} + 1 & \text{and} & \mathbf{t}_i \defeq \mathbf{t}_{i-1}, & \text{while}\\
					\text{if} & \operatorname{ind}(\mathbf{F}_i) \leq \Delta^{1/20}, & \text{then} & \mathbf{s}_i \defeq \mathbf{s}_{i-1} & \text{and} & \mathbf{t}_i \defeq \mathbf{t}_{i-1} + 1. &
					\end{array}
				\]
				
				\item Let $\mathbf{S}_i$ be a uniformly random independent set in $\mathbf{F}_i$ and let $\mathbf{I}_i \defeq (\mathbf{I}_{i-1} \setminus \Lambda_i) \cup \mathbf{S}_i$.
			\end{itemize}
			
			\noindent Set $\mathbf{I}' \defeq \mathbf{I}_k$.
		\end{leftbar}
		
		\noindent It is clear from \eqref{eq:flat} that the set $\mathbf{I}'$ constructed by the above procedure is uniformly distributed over the independent subsets of $L_J(N_G(u))$ (see also \cite[Lemma 15, Claim 1]{Mol17}).
		
		Let $\mathbf{a}(1)$, $\mathbf{a}(2)$, \ldots{} and $\mathbf{b}(1)$, $\mathbf{b}(2)$, \ldots{} be two infinite random sequences of zeros and ones drawn independently from each other, such that for all $s$ and $t$, we have
		\[
			\Pr [\mathbf{a}(s) = 1] = 1/2 \qquad \text{and} \qquad \Pr[\mathbf{b}(t) = 1] = \Delta^{-1/20}.
		\]
		Note that if the values $\mathbf{I}_0$, $\mathbf{S}_1$, \ldots, $\mathbf{S}_{i-1}$ are fixed, then the corresponding conditional probability of $\event{|\mathbf{S}_i| \geq \overline{\alpha}(\mathbf{F}_i)}$ is at least $1/2$, while the conditional probability of $\event{\mathbf{S}_i = \0}$ is precisely $1/\operatorname{ind}(\mathbf{F}_i)$ (here we are using the fact that the sets $\mathbf{I}_0$, $\mathbf{S}_1$, \ldots, $\mathbf{S}_{i-1}$ fully determine $\mathbf{F}_i$). Therefore, we can couple the distributions of the sequences $\mathbf{a}(1)$, $\mathbf{a}(2)$, \ldots{} and $\mathbf{b}(1)$, $\mathbf{b}(2)$, \ldots{} with the randomized procedure described above in such a way that 
		\begin{equation}\label{eq:coupling}
			\begin{array}{c c c c c c c}
			\text{if} & \operatorname{ind}(\mathbf{F}_i) > \Delta^{1/20} & \text{and} & \mathbf{a}(\mathbf{s}_i) = 1, & \text{then} & |\mathbf{S}_i| \geq \overline{\alpha}(\mathbf{F}_i), & \text{while} \\
			\text{if} & \operatorname{ind}(\mathbf{F}_i) \leq \Delta^{1/20} & \text{and} & \mathbf{b}(\mathbf{t}_i) = 1, & \text{then} & \mathbf{S}_i = \0. &
			\end{array}
		\end{equation}
		The Chernoff bound for independent random variables implies that, with probability at least $1 - \Delta^{-3}/8$,
		\begin{equation}\label{eq:ab_counts}
			|\set{1 \leq s \leq k/2 \,:\, \mathbf{a}(s) = 1}| \geq k/5 \qquad \text{and} \qquad |\set{1 \leq t \leq k/2 \,:\, \mathbf{b}(t) = 1}| \geq \ell.
		\end{equation}
		We claim that $|L_\mathbf{I}(u)| \geq \ell$ whenever \eqref{eq:ab_counts} holds. Since $\mathbf{s}_k + \mathbf{t}_k = k$, we always have either $\mathbf{s}_k \geq k/2$ or $\mathbf{t}_k \geq k/2$. If $\mathbf{s}_k \geq k/2$, then \eqref{eq:coupling} and the first part of \eqref{eq:ab_counts} imply that there are at least $k/5$ indices $i$ such that $\operatorname{ind}(\mathbf{F}_i) > \Delta^{1/20}$ and $|\mathbf{S}_i| \geq \overline{\alpha}(\mathbf{F}_i)$. By Lemma~\ref{lemma:She}, any such $i$ satisfies
		\[
			|\mathbf{S}_i| \geq \overline{\alpha}(\mathbf{F}_i) \geq f(\operatorname{ind}(\mathbf{F}_i)) \geq f(\Delta^{1/20}) > \frac{\log_2 \Delta}{40r \log_2 \log_2 \Delta},
		\]
		so in this case
		\[
			|\mathbf{I}'| = \sum_{i = 1}^k |\mathbf{S}_i| > \frac{k}{5} \cdot \frac{\log_2 \Delta}{40r \log_2 \log_2 \Delta} \geq \Delta.
		\]
		This is a contradiction, as $|\mathbf{I}'| \leq \deg_G(u) \leq \Delta$. Thus, we must have $\mathbf{t}_k \geq k/2$. From \eqref{eq:coupling} and the second part of~\eqref{eq:ab_counts}, we obtain that there are at least $\ell$ indices $i$ such that $\operatorname{ind}(\mathbf{F}_i) \leq \Delta^{1/20}$ and $\mathbf{S}_i = \0$. But $x_i \in L_{\mathbf{I}}(u)$ for any such $i$, so $|L_\mathbf{I}(u)| \geq \ell$, as desired. This completes the proof of~\ref{item:aKr}.
		
		To prove~\ref{item:bKr}, consider any collection $v_1$, \ldots, $v_{\lceil \ell \rceil}$ of $\lceil \ell \rceil$ distinct elements of $N_G(u)$. We claim that
		\begin{equation}\label{eq:factorial}
			\Pr\left[v_t \not \in \dom(\mathbf{I}') \text{ and } |L_\mathbf{I}(v_t)| \geq \ell \text{ for all } 1 \leq t \leq \lceil \ell \rceil\right] \leq \frac{1}{\lceil \ell \rceil!},
		\end{equation}
		which is enough as ${\Delta \choose \lceil \ell \rceil}/\lceil \ell \rceil! < \Delta^{-3}/8$ for large $\Delta$. To show~\eqref{eq:factorial}, consider an arbitrary independent set $Q \subseteq L_J(N_G(u))$ disjoint from $L(v_t)$ for all $1 \leq t \leq \lceil\ell\rceil$. We either have \[|L_J(v_t) \setminus N_H(Q)| < \ell \text{ for some } t,\] or else, there exist at least $\lceil \ell \rceil!$ ways to greedily choose elements $x_t \in L_J(v_t)$ so that $Q \cup \set{x_1, \ldots, x_{\lceil \ell \rceil}}$ is an independent set. Therefore,
		\[
			\Pr\left[v_t \not \in \dom(\mathbf{I}') \text{ and } |L_\mathbf{I}(v_t)| \geq \ell \text{ for all } 1 \leq t \leq \lceil \ell \rceil \,\middle\vert\, \mathbf{I}'\setminus L(\set{v_1, \ldots, v_{\lceil \ell \rceil}}) = Q\right] \leq \frac{1}{\lceil \ell \rceil!}.
		\]
		Since $Q$ is arbitrary, this yields~\eqref{eq:factorial}.
	\end{proof}

	\appendix
	
	\section*{Appendix: Proof of a weaker version of Lemma~\ref{lemma:small_deg}}
	
	\begin{lemma*}
		Let $\Cov{H} = (L, H)$ be a cover of a graph $G$. If there is a positive integer~$\ell$ such that $|L(u)| \geq \ell$ for all $u \in V(G)$ and $\deg_{\Cov{H}}^\ast(x) \leq \ell/8$ for all $x \in V(H)$, then $G$ is $\Cov{H}$\=/colorable. 
	\end{lemma*}
	\begin{proof}
		%Let $\Cov{H} = (L, H)$ be a cover of a graph $G$ and let $\ell$ be a positive integer such that $|L(u)| \geq \ell$ for all $u \in V(G)$ and $\deg_{\Cov{H}}^\ast(x) \leq \ell/8$ for all $x \in V(H)$.
		If necessary, we may remove some vertices from $H$ to arrange that %the cover $\Cov{H}$ is $\ell$-fold.
		arrange that $|L(u)| = \ell$ for all $u \in V(G)$.
	Let $\mathbf{I}$ be a random subset of $V(H)$ obtained by choosing, independently and uniformly, a single vertex from each list $L(u)$. For $xy \in E(H^\ast)$, let $B_{xy}$ denote the random event $\event{\set{x, y} \subseteq \mathbf{I}}$. Then $\mathbf{I}$ is an independent set, and hence an $\Cov{H}$\=/coloring, precisely when none of the events $B_{xy}$ happen.
	By definition, $\Pr\left[B_{xy}\right] = \ell^{-2}$. Let $u$, $v \in V(G)$ be such that $x \in L(u)$, $y \in L(v)$ and define
	\[
	\Gamma(xy) \defeq \set{x'y' \in E(H^\ast) \,:\, x' \in L(u) \text{ or } y' \in L(v)}.
	\] 
	Then
	\[\label{eq:degree}
	|\Gamma(xy)| \leq \sum_{x' \in L(u)} \deg_{\Cov{H}}^\ast(x') + \sum_{y' \in L(v)} \deg_{\Cov{H}}^\ast(y') \leq 2 \cdot \ell \cdot \ell/8 =  \ell^2/4.
	\]
	Since $B_{xy}$ is mutually independent from the events $B_{x'y'}$ with $x'y' \not \in \Gamma(xy)$, an application of the~LLL proves that $\mathbf{I}$ is an $\Cov{H}$\=/coloring of~$G$ with positive probability, as desired.
	\end{proof}

	\subsection*{Acknowledgments} I am very grateful to the anonymous referees for their valuable comments.
	
	{
	\printbibliography}

\end{document}